\documentclass[11pt,  reqno]{amsart}

\usepackage{amsmath,amssymb,amscd,amsthm,amsxtra, esint}

\allowdisplaybreaks[2]

\newtheorem{theorem}{Theorem} [section]

\newtheorem{lemma}[theorem]{Lemma}

\newtheorem{remark}[theorem]{Remark}

\newtheorem{corollary}[theorem]{Corollary}

\newcommand{\noi}{\noindent}

\newcommand{\R}{\mathbb{R}}

\newcommand{\T}{\mathbb{T}}

\newcommand{\N}{\mathcal{N}}
\newcommand{\RR}{\mathcal{R}}

\newcommand{\al}{\alpha}

\newcommand{\eps}{\varepsilon}

\newcommand{\ld}{\lambda}
\newcommand{\Ld}{\Lambda}

\newcommand{\ft}{\widehat}
\newcommand{\wt}{\widetilde}

\newcommand{\dx}{\partial_x}
\newcommand{\dt}{\partial_t}

\newcommand{\jb}[1]
{\langle #1 \rangle}

\numberwithin{equation}{section}
\numberwithin{theorem}{section}

\begin{document}


\date{\today}

\title
[Unconditional well-posedness of mKdV]
{On unconditional  well-posedness  of modified KdV }

\author{Soonsik Kwon and Tadahiro Oh}

\address{
Soonsik Kwon\\
Department of Mathematical Sciences\\
Korea Advanced Institute of Science and Technology\\
335 Gwahangno \\ Yuseong-gu, Daejeon 305-701, Republic of Korea}

\email{soonsikk@kaist.edu}

\address{Tadahiro Oh\\
Department of Mathematics\\
University of Toronto\\
40 St. George St, Toronto, ON M5S 2E4, Canada}

\curraddr{Department of Mathematics\\
Princeton University\\
Fine Hall, Washington Rd\\
Princeton, NJ 08544-1000, USA}

\email{hirooh@math.princeton.edu}

\subjclass[2000]{35Q53}

\keywords{modified KdV; well-posedness; uniqueness}

\thanks{The first author is supported by NRF grant 2010-0024017.}

\begin{abstract}
Bourgain \cite{BO1} proved that the periodic modified KdV equation (mKdV)
is locally well-posed in $H^s(\T)$, $s \geq 1/2$,
by introducing new weighted Sobolev spaces $X^{s, b}$, where the uniqueness holds
{\it conditionally}, namely in $C([0, T]; H^s) \cap X^{s, \frac{1}{2}}([0,  T]\times \T)$.
In this paper, we establish {\it unconditional} well-posedness of mKdV in $H^s(\T)$, $s \geq \frac{1}{2}$,
i.e. in addition we establish {\it unconditional} uniqueness in $C([0, T]; H^s)$, $s \geq 1/2$,
of solutions to mKdV. We prove this result via differentiation by parts.
For the endpoint case $s = \frac{1}{2}$,
we perform careful quinti- and septi-linear estimates after the second differentiation by parts.

\end{abstract}

\maketitle

\tableofcontents

\section{Introduction}

We consider the Cauchy problem for the modified Korteweg-de Vries (mKdV) equation
on the one-dimensional torus $\T = \R /2\pi \mathbb{Z}$:
\begin{align} \label{MKDV1}
\begin{cases}
\dt u = \dx^3 u \pm u^2 \dx u, \\
u\big|_{t = 0} = u_0,
\end{cases}
\quad (x, t) \in \T \times \R
\end{align}

\noi
where $u$ is a real-valued function.
The mKdV has received a great deal of attention
both from applied and theoretical fields
and is known to be completely integrable in the sense that
it enjoys the Lax pair structure and so infinitely many conservation laws.
In particular, if $u$ is a ``nice'' solution of \eqref{MKDV1},
then the $L^2$-norm is conserved.
i.e. $\|u(t)\|_{L^2} = \|u_0\|_{L^2}$.
Then, by the change of variables $x \to x\mp\mu t$ with $\mu = \frac{1}{2\pi}\|u_0\|^2_{L^2}$,
we can rewrite \eqref{MKDV1} as
\begin{align} \label{MKDV2}
\begin{cases}
\dt u = \dx^3 u \pm \Big(u^2 - \tfrac{1}{2\pi} \int_\T u^2\Big)\dx u, \\
u\big|_{t = 0} = u_0.
\end{cases}
\end{align}

\noi
In this paper, we study the unconditional local well-posedness of \eqref{MKDV2}.

Let us briefly go over recent results on the well-posedness theory
of the periodic mKdV.
In \cite{BO1}, Bourgain introduced a new weighted space-time Sobolev space $X^{s,b}$ (also known as {\it dispersive-Sobolev space}), whose norm
is given by
\begin{equation} \label{XSB}
 \|u\|_{X^{s, b}(\T\times\R)} = \|\jb{k}^s \jb{\tau+k^3}^b \ft{u}(k, \tau)\|_{l^2_k L^2_\tau(\mathbb{Z}\times\R)},
\end{equation}

\noi
where $\jb{ \, \cdot \, } = 1 +|\cdot|$.
By the fixed point argument in an appropriate $X^{s, b}$ space,
he proved that \eqref{MKDV2} is locally well-posed in $H^s(\T)$, $s\geq \frac{1}{2}$.
Then, Colliander-Keel-Staffilani-Takaoka-Tao \cite{CKSTT4}
proved global well-posedness in $H^s$, $s \geq \frac{1}{2}$, via the $I$-method.
We point out that the solution map $\mathcal{S}_t: u_0 \in H^s \mapsto u(t)\in H^s$
constructed in \cite{BO1, CKSTT4}
is smooth.
Indeed,
it was shown in \cite{BO2} that the solution map to \eqref{MKDV2}
can not be smooth in $H^s$ for $s < \frac{1}{2}$.
See \cite{CCT, TT} for related results.
Nonetheless, Takaoka-Tsutsumi \cite{TT} successfully modified the $X^{s, b}$ space
to reduce the nonlinear effect from the resonant term
(see $\RR$ in \eqref{v1} below) and proved local well-posedness in $H^s$, $s >\frac{3}{8}$.
Nakanishi-Takaoka-Tsutsumi \cite{NTT} further improved the result
and proved local well-posedness in $H^s \cup \mathcal{F}L^{\frac{1}{2}, \infty}$, $s> \frac{1}{3}$,
where $\|f\|_{\mathcal{F}L^{\frac{1}{2}, \infty}} = \sup_k \jb{k}^\frac{1}{2} |\ft{f}(k)| <\infty$.
(Existence alone holds in $H^s$ for $s>\frac{1}{4}$.)
Note that the solution map constructed in \cite{TT, NTT} is not uniformly continuous.
There is also a result using the complete integrability of the equation.
Kappeler-Topalov \cite{KT1} proved that
defocusing mKdV, \eqref{MKDV1} with the $-$ sign,
is globally well-posed in $L^2(\T)$ via the inverse spectral method.

Now, let us examine the uniqueness of solutions in the above results.
In \cite{BO1, TT}, the uniqueness (with prescribed $L^2$-norm) holds in (a ball in)
$C([0,T];H^s)\cap X$, where $X$ is an auxiliary function space,
i.e. only within the (modified) $X^{s, b}$ space.
Thus,  the uniqueness holds {\it conditionally},
since uniqueness may not hold without the restriction of the auxiliary space $X$.
In \cite{KT1},
the uniqueness holds in the class of solutions obtained by a limiting procedure of smooth solutions.

Recall the following definition from Kato \cite{KATO}.
We say that a Cauchy problem is {\it unconditionally well-posed}
in $H^s$
if for every initial condition $u_0 \in H^s$,
there exist $T>0$ and a {\it unique} solution $u \in C([0, T];H^s)$
such that $u(0) = u_0$.
Also, see \cite{FPT}.
We refer to such uniqueness in $C([0, T];H^s)$
without intersecting with any auxiliary function space as {\it unconditional uniqueness}.
Unconditional uniqueness is a concept of uniqueness which does not depend
on how solutions are constructed.
See, for example,  Zhou \cite{Z} for unconditional uniqueness of KdV in $L^2(\R)$\footnote{This is uniqueness of weak solutions in the space of $L^\infty_tL^2_x(\R\times\R)$} or Tao \cite{Ta}
for focusing mass-critical NLS with spherical symmetry in $L^2(\R^d)$, $d \geq 5$.

The main result of the paper is the following:
\begin{theorem} \label{thm1}
Let $s\geq \frac{1}{2}$.
Then,  mKdV is unconditionally locally well-posed
in $H^s(\T)$.
\end{theorem}

\noi Our result provides another proof of the local well-posedness.
We think that this proof is more natural and elementary since we do not use any auxiliary function spaces
but only rely on simple {\it differentiation by parts} and {\it Cauchy-Schwarz inequality}. As a result, we can establish unconditional uniqueness of solution to mKdV in $H^{s}(\T)$, $s \geq\frac{1}{2}$, which is an improvement of Bourgain's result \cite{BO1} in the aspect of uniqueness.

\begin{remark}\rm
In the proof of Theorem~\ref{thm1}, we show existence and uniqueness of solutions to 
the renormalized mKdV \eqref{MKDV2} in $H^s(\T)$, $s\geq \frac{1}{2}$. From this, one can deduce existence and uniqueness of solutions to the original mKdV \eqref{MKDV1} in $H^s(\T)$, $s\geq \frac{1}{2}$. 
Indeed, for a given $u_0 \in H^s(\T)$ with $s \geq \frac{1}{2}$, 
a function  $u(x,t) \in L^\infty([0,T];H^s)$ is a solution to the original mKdV \eqref{MKDV1} with the initial condition $u_0$
if and only if 
$v$ is a solution to the renormalized mKdV \eqref{MKDV2} with the same initial condition $u_0$,
where $v$ is defined by
\begin{equation} \label{gauge}
v(x,t):= u\big(x\mp \frac{1}{2\pi} \int_0^t \|u(t')\|_{L^2}^2dt' ,t\big). 
\end{equation}  

\noi
(Here, we used the fact that $\|v(t)\|_{L^2} = \|u(t)\|_{L^2}$
for $v$ defined in \eqref{gauge}.)
 
Now, suppose that $u_1$ and $u_2$ are two solutions to the original mKdV \eqref{MKDV1} 
in $C([0, T]; H^s)$
with the same initial condition 
$u_0 \in H^s(\T)$ with $s \geq \frac{1}{2}$.
Then, $t \mapsto \|u_j(t) \|_{L^2}^2  $ is locally integrable and $v_1$ and $v_2$ defined via \eqref{gauge} are solutions to the renormalized mKdV \eqref{MKDV2} in $C([0, T]; H^s)$ with the same initial condition $u_0$. 
Hence, by Theorem \ref{thm1}, we have $v_1 = v_2$ in $C([0, T]; H^s)$. 
In particular, $\|v_j(t) \|_{L^2}^2 $, $j = 1, 2$, is constant in time.\footnote{For (unique) solutions to \eqref{MKDV2} 
in $C([0, T]; H^s)$, $s\geq \frac{1}{2}$, 
constructed in this paper, one can easily check that the $L^2$-norm is conserved in time.}
In view of \eqref{gauge}, 
we see that  $\|u_j(t)\|_{L^2}^2$, $j = 1, 2$, is also constant in time, and
the transformation \eqref{gauge} can be written as
\begin{equation} \label{gauge2} 
v_j(x,t)= u_j\big(x\mp \tfrac{t}{2\pi} \|u_0\|_{L^2}^2,t\big),
\quad \text{ for }j = 1, 2.
\end{equation}

\noi
Therefore, from (the inverse of) \eqref{gauge2} and $v_1 = v_2$ in $C([0, T]; H^s)$, 
 we obtain $u_1 = u_2$ in $C([0, T]; H^s)$. This shows unconditional uniqueness of the original mKdV \eqref{MKDV1}.

Lastly, we discuss the regularity of the solution map$: u_0 \mapsto u(t)$
of the original mKdV \eqref{MKDV1} (for sufficiently small $t$ depending on the size of initial data.) 
From the proof of Theorem \ref{thm1},
it follows that the solution map of the renormalized mKdV \eqref{MKDV2} is locally Lipschitz continuous.
Consequently,  this yields
 local Lipschitz continuity of the solution map of the original mKdV \eqref{MKDV1}
  in the class
 \begin{equation} \label{L2}
 \{u_0 \in H^s(\T) :  \|u_0\|_{L^2(\T)}=c   \}
 \end{equation}
 
 \noi
 with a fixed $c$. 
(Two initial data of distinct $L^2$-norms give rise to two different
renormalized mKdV \eqref{MKDV2}, and thus their solutions are not comparable.
In general, one can show that the uniform continuity of
the solution map of the original mKdV \eqref{MKDV1} fails without
prescribing the $L^2$-norm.)
\end{remark}

\begin{remark}\rm
Many of the unconditional uniqueness results
use some auxiliary function spaces (e.g. $X^{s, b}$ spaces in \cite{Z}, Strichartz spaces in \cite{Ta}),
which are designed to be large enough to contain $C([0, T];H^s)$
such that desired nonlinear estimates hold.
However,  we simply use
the $C([0, T];H^s)$-norm in the proof of Theorem \ref{thm1}.

We also point out that Theorem \ref{thm1} does {\it not} imply
unconditional well-posedness for KdV in $H^s(\T)$, $s \geq -\frac{1}{2}$
even under the Miura map.
Indeed, the issue of unconditional uniqueness of the periodic KdV is settled
in view of the non-uniqueness result by Christ \cite{CHRIST} for KdV in $C_tH^s(\T)$, $s < 0$,
and an (implicitly implied) positive result in $L^2(\T)$ by
Babin-Ilyin-Titi \cite{BIT}.
\end{remark}

Theorem \ref{thm1} with global well-posedness of mKdV in
$H^s$, $s \geq \frac{1}{2}$, by \cite{CKSTT4}
yields  the following corollary.

\begin{corollary} \label{cor1}
Let $s\geq \frac{1}{2}$.
Then,  mKdV is unconditionally globally well-posed
in $H^s(\T)$.
\end{corollary}

\noi
We prove Theorem \ref{thm1}
by establishing {\it a priori} estimates,
where we use only the $C_tH^s_x$-norm of solutions.
In the following, we briefly describe the idea of
differentiation by parts introduced in Babin-Ilyin-Titi \cite{BIT}.

Let $S(t) = e^{t\dx^3}$ denote the semigroup to the Airy equation
(= linear part of mKdV \eqref{MKDV1}.)
We apply a change of coordinates: $v(t) = S(-t) u(t)$.
In terms of the spatial Fourier coefficients, this can be written as
$v_k (t) = e^{ik^3t}u_k(t)$,
where $v_k(t)$ denotes the $k$-th (spatial) Fourier coefficient of $v(\cdot, t)$.
i.e. $v_k (t)= \ft{v}(k,t)$.
Working in terms of $v$ has certain advantages.
Ginibre \cite{G} says
``In the language of Quantum Mechanics, this consists in working in the so-called interaction representation.''
In \cite{BO1}, Bourgain made an effective use of this coordinate
by introducing the $X^{s, b}$ spaces.
From the definition \eqref{XSB}, we have
$\|u\|_{X^{s, b}} = \|v\|_{H^b_tH^s_x}$,
i.e. a function $u$ is in $X^{s, b}$
if and only if its interaction representation $v(t) = S(-t) u(t)$
is in the classical Sobolev space  $H^b_t H^s_x$.
A similar idea has been applied to study equations in hydrodynamics.
See \cite[Section 2]{BIT} for a nice discussion.


With $v (t) = S(-t)u(t)$,
it follows from \eqref{MKDV2}
(see \cite{BO1})
that $v$ satisfies\footnote{In the following, we only deal with the focusing case, i.e. with the $+$
sign in \eqref{MKDV1}, since our analysis is local-in-time
and thus the focusing/defocusing nature of the equation is irrelevant.}
\begin{align} \label{v1}
\dt v_k & = \sum_{\substack{k_1+ k_2+ k_3 = k\\k_1+k_2\ne0}} i k_3 e^{i t\Phi(\bar{k})} v_{k_1}v_{k_2}v_{k_3} \notag \\
& = \frac{i}{3}\sum_{\substack{k_1+ k_2+ k_3 = k\\ \Phi(\bar{k}) \ne 0}} k e^{i t\Phi(\bar{k})} v_{k_1}v_{k_2}v_{k_3}
- i k |v_k|^2 v_k
=: \N + \RR.
\end{align}

\noi
where $\Phi(\bar{k}) =  \Phi(k, k_1, k_2, k_3) := k^3 - k_1^3 - k_2^3 - k_3^3$.
With $k = k_1 + k_2+k_3$, we have
\begin{equation} \label{PHI}
\Phi(\bar{k}) = 3 (k_1 + k_2) (k_2+k_3) (k_3+k_1).
\end{equation}

\noi
In this framework, the usual Duhamel formulation of \eqref{MKDV2}
corresponds to
\begin{equation} \label{S1}
v(t) = v_0 + \int_0^t \N(v)(t') + \RR(v)(t') dt',
\end{equation}

\noi
where $\N(t')$ stands for $\N(v)(t') = \N\big(v(t'),v(t'),v(t')\big)$, and so on.
Due to the presence of $\dx$ in the nonlinearity,
a direct estimate in $H^{s}$ on the nonlinear part in \eqref{S1}
does not work.

Assume that $v$ is smooth in the following,
since our goal is to obtain a priori estimates on solutions.
As in \cite{BIT},
we can differentiate $\N$ by parts,
Then, we have
\begin{align} \label{v2}
\N_k & = \dt \bigg[\frac{i}{3}\sum_{\substack{k_1+ k_2+ k_3 = k\\\Phi(\bar{k}) \ne 0}}
\frac{k e^{i t\Phi(\bar{k})}}{i \Phi(\bar{k})} v_{k_1}v_{k_2}v_{k_3}\bigg] \notag\\
& \hphantom{X} - \frac{i}{3} \sum_{\substack{k_1+ k_2+ k_3 = k\\ \Phi(\bar{k}) \ne 0}}
\frac{k e^{i t\Phi(\bar{k})}}{i \Phi(\bar{k})}
( \dt v_{k_1}v_{k_2}v_{k_3}+v_{k_1}\dt v_{k_2}v_{k_3}+ v_{k_1}v_{k_2}\dt v_{k_3})\\
& =: \dt (\N_1)_k + (\N_2)_k. \notag
\end{align}

\noi
Note that this  corresponds to integration by parts on $\int \N(t) dt$.
With \eqref{v2}, we see that smooth solutions to \eqref{v2} satisfy
\begin{equation} \label{S2}
v(t) = v_0 + \N_1(t)-\N_1(0) + \int_0^t \N_2(t') + \RR(t') dt'.
\end{equation}

\noi
In \eqref{v2}, both terms have $\Phi(\bar{k}) (\ne 0)$ in the denominators,
and this provides smoothing.

%
%
%

Now, suppose that
we have $C_t H^s_x$ estimates on $\N_1$, $\N_2$, and $\RR$
in \eqref{S2}.
For $\N_2$ and $\RR$, we obtain smallness thanks to the time integration (for small $t$.)
However, there is no small constant for $\N_1$.
Thus, we can not close the argument to obtain a contraction.

In order to fix this problem, we use the idea from Section 6 in \cite{BIT}.
The idea is to separate the low frequency part of the non-resonant part $\N$
before differentiating by parts.
Let $v^{(n)} = P_n v$, where $P_n$ is the Dirichlet projection
onto the frequencies $\{\,|k| \leq n\}$.
Then, write $\N = \N^{(n)} + \N^{(-n)}$,
where $\N^{(n)}$ is given by
\begin{equation} \label{v4}
\N^{(n)}
= \frac{i}{3}\sum_{\substack{k_1+ k_2+ k_3 = k\\ \Phi(\bar{k}) \ne 0}} k e^{i t\Phi(\bar{k})}
v^{(n)}_{k_1}v^{(n)}_{k_2}v^{(n)}_{k_3}
\end{equation}

\noi
and $\N^{(-n)} = \N- \N^{(n)}  $.
Differentiating $\N^{(-n)}$ by parts,
we obtain
\begin{align} \label{v5}
\N^{(-n)} = \dt \big(\N_1^{(-n)}\big)+  \N_2^{(-n)},
\end{align}

\noi
where $\N_1^{(-n)}$ and $\N_2^{(-n)}$ are as in \eqref{v2}
with an extra condition
\begin{equation} \label{v55}
 k^* := \max (|k_1|,|k_2|,|k_3|) > n.
\end{equation}

\noi
Hence, smooth solutions to \eqref{v2} satisfy
\begin{align} \label{v6}
v(t) & =v_0+ \N^{(-n)}_{1}(t)-\N_{1}^{(-n)}(0)
+ \int_0^t \RR(t')
+ \N^{(n)}(t')+ \N_{2}^{(-n)}(t')dt'.
\end{align}

\noi
It turns out that \eqref{v55} provides a small constant
$n^{-\al}$ for some $\al>0$
in estimating $\N_{1}^{(-n)}$ (see Lemma \ref{LEM:N3} below),
and we can close the argument for $s > \frac{1}{2}$.

Unfortunately, this turns out not to be sufficient when $s = \frac{1}{2}$.
In particular, the estimate on $\N_{2}^{(-n)}$
fails when $ s= \frac{1}{2}$.
(See Lemma \ref{LEM:N2} below.)
Thus, we need to proceed one step further.
 We hoped to take a differentiation by parts once more.
 A direct differentiation by parts, however,  does not work because the corresponding resonance function
 $\Phi(\bar{k}) + \Phi(\bar{j})$ does not have a good factorization. (See \eqref{N23} below.) In this case, we restrict $\N_{2}^{(-n)}$ into a part and then perform differentiation by parts
once more, but in a slightly more complicated manner.
See \eqref{O17} and \eqref{O333}.
Namely, in \eqref{v2} and \cite{BIT},
we perform differentiation by parts to simply move the time derivative
from a complex exponential to a product of $v_{k_j}$.
However, in \eqref{O333},
we need to perform integration  by parts\footnote{Indeed, we keep \eqref{O333} in the form of
integration by parts to emphasize this point.} to move the time derivative
$\N_1 (=$ product of $e^{it\Phi(\bar{\jmath})}$ and $v_{j_1}v_{j_2}v_{j_3}$)
to $e^{it\Phi(\bar{k})}$ {\it and} $v_{k_2}v_{k_3}$,
which leads to further quinti- and septi-linear estimates.
See Section \ref{SEC:endpoint} for details.

Lastly, we point out that the restriction $s \geq \frac{1}{2}$
on the regularity
is due to the resonant term $\RR$ (see \cite[p.228]{BO1} and Lemma \ref{LEM:R}.)
As pointed out in \cite{TT}, if we define $v$ by
\begin{equation}\label{S4}
v_k (t)=  e^{i k^3 t +  i k \int_0^t |u_k(t')|^2 t' }u_k(t),
\end{equation}

\noi
then this would formally eliminate the resonant term.
However, it is difficult to make sense of this transformation for nonsmooth functions.
Instead, following \cite{TT, NTT}, one may try to use
\begin{equation}\label{S5}
v_k (t)=  e^{i k^3 t +  i k |u_k(0)|^2 t }u_k(t)
\end{equation}

\noi
as the first order approximation to \eqref{S4}
in order to weaken the nonlinear effect of the resonant term $\RR$.
For further improvement, one may consider the second order approximation:
$v_k (t)=  e^{i k^3 t +  i k (|u_k(0)|^2 t + \frac{1}{2} \dt|u_k(0)|^2 t^2) }u_k(t) $
or
higher order approximations.

\medskip

This paper is organized as follows.
In Section 2, we prove a priori estimates needed for $s > \frac{1}{2}$.
Then, we present the proof of Theorem \ref{thm1} for $s > \frac{1}{2}$ in Section 3.
In Section 4, we present the argument for the endpoint case $s = \frac{1}{2}$.

\medskip
\noindent
{\bf Acknowledgments:}
The authors would like to thank the anonymous referee for helpful suggestions.

\section{Nonlinear estimates for $s > \frac{1}{2}$}

In this section, we present nonlinear estimates controlling the terms in \eqref{v6}.
Without loss of generality, we assume that $v_k$ is nonnegative
in the following.

\begin{lemma} \label{LEM:R}
Let $\RR$ be as in \eqref{v1}.
Then, for any $s \in\R$, we have
\begin{align} \label{R1}
\|\RR(v)\|_{H^s} & \lesssim \|v\|^2_{H^{\frac{1}{2}}}\|v\|_{H^s}.
\end{align}

\noi
Also, for  $s \geq \frac{1}{2}$, we have
\begin{align}\label{R2}
\|\RR(v) - \RR(w) \|_{H^s}
& \lesssim \big(\|v\|_{H^s}+ \|w\|_{H^s}\big)^2 \|v-w\|_{H^s}.
\end{align}

\end{lemma}

\begin{proof}
We only prove \eqref{R1} since \eqref{R2} follows in a similar manner.
Clearly, we have
\begin{align*}
\|\RR(v)\|_{H^s}
= \bigg(\sum_k |k|^{2+ 2s} v_k^6 \bigg)^\frac{1}{2}
\leq \big\||k|^\frac{1}{2} v_k \big\|_{l^\infty_k}^2\|v\|_{H^s},
\end{align*}

\noi
which is bounded by RHS of \eqref{R1}.
\end{proof}

In the following, fix $n \in \mathbb{N}$.

\begin{lemma} \label{LEM:N1}
Let $\N^{(n)}$ be as in \eqref{v4}.
Then, for $s\geq \frac{1}{2}$, we have
\begin{align} \label{N11}
\|\N^{(n)}(v)\|_{H^s} & \lesssim n \ln n \, \|v\|_{H^s}^3 \\
\label{N12}
\|\N^{(n)}(v) - \N^{(n)}(w) \|_{H^s}
& \lesssim n \ln n\,  \big(\|v\|_{H^s}+ \|w\|_{H^s}\big)^2 \|v-w\|_{H^s}.
\end{align}

\begin{proof}
We only prove \eqref{N11} since \eqref{N12} follows in a similar manner.
Without loss of generality, assume $|k_1| \gtrsim |k|$.
Then, by $|k| \lesssim n $ and Young's inequality, we have
\begin{align*}
\|\N^{(n)}(v)\|_{H^s}
& \leq
\bigg(\sum_k |k|^{2+2s} \Big( \sum_{k_1+k_2+k_3=k}
v^{(n)}_{k_1}v^{(n)}_{k_2}v^{(n)}_{k_3} \Big)^2 \bigg)^\frac{1}{2}\\
& \lesssim n
\bigg(\sum_k  \Big( \sum_{k_1+k_2+k_3=k}
|k_1|^s v^{(n)}_{k_1}v^{(n)}_{k_2}v^{(n)}_{k_3} \Big)^2 \bigg)^\frac{1}{2}\\
& \leq n \|v^{(n)}\|_{H^s} \|v^{(n)}_{k_2}\|_{l^1_{k_2}} \|v^{(n)}_{k_3}\|_{l^1_{k_3}}
\lesssim n \ln n \, \|v^{(n)}\|_{H^s} \|v^{(n)}\|^2_{H^\frac{1}{2}}
\end{align*}

\noi
where we used Cauchy-Schwarz inequality and $|k_j|\leq n$ in the last step.
\end{proof}

\end{lemma}

Recall the following \cite[(8.21), (8.22)]{BO1}:
Suppose $\Phi(\bar{k}) \ne 0$ when $k = k_1 + k_2 + k_3$.
Then,  we have the following two possibilities:
\begin{itemize}
\item[(a)] With $ k^* = \max (|k_1|, |k_2|, |k_3|)$,
\begin{align}
 |\Phi (k)|
&  \geq
 \max (|k_1 + k_2||k_2 + k_3|,|k_2 + k_3||k_3 + k_1|,|k_3 + k_1||k_1 + k_2|) \notag \\
& \gtrsim
(k^*)^2, \quad \label{PHI1}
\end{align}

\noi
In this case, we have $ |\Phi (k)| \gtrsim (k^*)^2 \ld$,
where
\begin{align} \ld = \ld_k:= \min(|k_1 + k_2|,|k_2 + k_3|,|k_3 + k_1|).     \label{LD1}
\end{align}

\item[(b)] $|k_1|\sim |k_2|\sim|k_3| \sim k^*$ \ and
\begin{align}
 |\Phi (k)| \geq &\max (|k_1 + k_2|,|k_2 + k_3|,|k_3 + k_1|)
 \gtrsim k^*. \label{PHI2}
\end{align}

\noi
In this case, we have
$ |\Phi (k)| \gtrsim k^*\Ld $,
where
\begin{align}
\label{LD2}
 \Ld = \Ld_k := \min &(|k_1 + k_2||k_2 + k_3|,|k_2 + k_3||k_3 + k_1|,|k_3 + k_1||k_1 + k_2|).
\end{align}
\end{itemize}

\begin{lemma} \label{LEM:N2}
Let $\N_2^{(-n)}$ be as in \eqref{v5}.
Then, for $s> \frac{1}{2}$, we have
\begin{align} \label{N21}
\|\N_2^{(-n)}(v)\|_{H^s} & \lesssim  \, \|v\|_{H^s}^5 \\
\label{N22}
\|\N_2^{(-n)}(v) - \N_2^{(-n)}(w) \|_{H^s}
& \lesssim   \big(\|v\|_{H^s}+ \|w\|_{H^s}\big)^4 \|v-w\|_{H^s}.
\end{align}

\noi
The same estimates hold for $\N_2$ in \eqref{v2}.
\end{lemma}

\begin{proof}
We only prove \eqref{N21} since \eqref{N22} follows in a similar manner.
From \eqref{v1} and \eqref{v2}, we can separate $\N_2^{(-n)}$ into two parts:
\begin{align} \label{N23}
\big(\N_2^{(-n)}\big)_k & = -\sum_{\substack{k_1+ k_2+ k_3 = k\\\Phi(\bar{k}) \ne 0 \\k^* > n}}
\frac{k k_1e^{i t\Phi(\bar{k})}}{i \Phi(\bar{k})} |v_{k_1}|^2 v_{k_1}v_{k_2}v_{k_3} \notag \\
& + \frac{1}{3}\sum_{\substack{k_1+ k_2+ k_3 = k\\
j_1 + j_2 +j_3 = k_1 \\ \Phi(\bar{k}), \Phi(\bar{\jmath}) \ne 0\\k^* > n}}
\frac{k k_1e^{i t(\Phi(\bar{k})+\Phi(\bar{\jmath}))}}{i \Phi(\bar{k})} v_{j_1}v_{j_2}v_{j_3}v_{k_2}v_{k_3}
=: \big(\N_{21}^{(-n)}\big)_k+\big(\N_{22}^{(-n)}\big)_k,
\end{align}

\noi
where $\Phi(\bar{\jmath}) :=  \Phi(k_1, j_1, j_2, j_3)$.

\medskip
\noi
$\bullet$ {\bf Part 1:}
First, we estimate $\N_{21}^{(-n)}$.
By duality, it suffices to prove
\begin{align} \label{N24}
 \sum_k \sum_{\substack{k_1+ k_2+ k_3 = k\\\Phi(\bar{k}) \ne 0 \\k^* > n}}
M_1 |u_{k_1}|^2 u_{k_1}u_{k_2}u_{k_3} z_k
\lesssim \|u\|_{L^2}^5
\end{align}

\noi
where $\|z\|_{L^2} = 1$ and $M_1 $ is given by
\[ M_1 = M_1(k_1, k_2, k_3) : = \frac{|k|^{1+s} |k_1|}{|\Phi(\bar{k})||k_1|^{3s} |k_2|^s|k_3|^s}.\]

\noi
By Cauchy-Schwarz inequality, we have
\begin{align*}
\text{LHS of }\eqref{N24}
& \leq \bigg(
\sum_{k_1, k_2, k_3}
|u_{k_1}|^6 |u_{k_2}|^2|u_{k_3}|^2\bigg)^\frac{1}{2}
\bigg( \sum_k \sum_{k_1 + k_2 +k_3 = k}   M_1^2 z_k^2\bigg)^\frac{1}{2}\\
& \leq
\mathcal{M}_1
\|u\|_{L^2}^5,
\end{align*}

\noi
where $\mathcal{M}_1 = \big( \sum_k \sum_{k_1 + k_2 +k_3 = k}   M_1^2 z_k^2\big)^\frac{1}{2}$.

\medskip
\noi
$\circ$ Case 1.a:  $\Phi(\bar{k})$  satisfies \eqref{PHI1}.
In this case, we have
\begin{align*}
M_1^2 \sim \frac{1}{(k^*)^{2-2s} |k_1|^{6s-2}|k_2|^{2s}|k_3|^{2s}
\ld^2}
\leq \frac{1}{(k_*)^{8s} \ld^2}
\end{align*}
\noi
where $k_* = \min (|k_1|, |k_2|, |k_3|)$.
Thus, for $ s > \frac{1}{8}$, we have
$\mathcal{M}_1 \lesssim 1$
by summing over $k_i \, ( \ne k_*)$ in $\ld^{-2}$,
$k_*$ for $(k_*)^{-8s}$, and then $k$ for $z_k^2$.
Hence, \eqref{N24} holds for $s > \frac{1}{8}$.

\medskip
\noi
$\circ$ Case 1.b:  $\Phi(\bar{k})$  satisfies \eqref{PHI2}.
In this case, we have
\begin{align*}
M_1^2 \sim \frac{1}{(k^*)^{8s-2}
\Ld^2}
\leq \frac{1}{(k_*)^{8s-2} \Ld^2}.
\end{align*}

\noi
Thus, for $ s \geq \frac{1}{4}$, we have
$\mathcal{M}_1 \lesssim 1$
by summing over two frequencies in $\Ld^{-2}$,
 and then $k$ for $z_k^2$.
 Hence, \eqref{N24} holds for $s \geq \frac{1}{4}$.

\medskip
Therefore,  the estimates \eqref{N21} and \eqref{N22} for $\mathcal{N}^{(-n)}_{21}$ hold as long as  $s\ge \frac 14$.

\medskip
\noi
$\bullet$ {\bf Part 2:}
Next, we estimate $\N_{22}^{(-n)}$.
By duality, it suffices to prove
\begin{align} \label{N25}
 \sum_k \sum_{\substack{k_1+ k_2+ k_3 = k\\
j_1 + j_2 +j_3 = k_1 \\ \Phi(\bar{k}) \ne 0\\k^* > n}}
 M_2 u_{j_1}u_{j_2}u_{j_3} u_{k_2}u_{k_3} z_k
\lesssim \|u\|_{L^2}^5
\end{align}

\noi
where $\|z\|_{L^2} = 1$ and $M_2 $ is given by
\[ M_2 = M_2(j_1, j_2, j_3, k_2, k_3)
: = \frac{|k|^{1+s} |k_1|}{|\Phi(\bar{k})||j_1|^s|j_2|^s|j_3|^s |k_2|^s|k_3|^s}.\]

\noi
By Cauchy-Schwarz inequality, we have
\begin{align*}
\text{LHS of }\eqref{N25}
& \leq \bigg(
\sum_{j_1,j_2, j_3, k_2, k_3}
|u_{j_1}|^2|u_{j_2}|^2|u_{j_3}|^2 |u_{k_2}|^2|u_{k_3}|^2\bigg)^\frac{1}{2}
\bigg( \sum_k
\sum_{\substack{k_1+ k_2+ k_3 = k\\
j_1 + j_2 +j_3 = k_1}}
M_2^2 z_k^2\bigg)^\frac{1}{2}\\
& \leq
\mathcal{M}_2
\|u\|_{L^2}^5,
\end{align*}

\noi
where $\mathcal{M}_2 =
\big( \sum_k
\sum_{\substack{k_1+ k_2+ k_3 = k\\
j_1 + j_2 +j_3 = k_1}}
M_2^2 z_k^2\big)^\frac{1}{2}$.
Without loss of generality, assume
\begin{align}\label{Z1}
|j_1| = \max (|j_1|, |j_2|, |j_3|)
\gtrsim |k_1|.\\
\intertext{Also, we assume}
|k_1| = \max (|k_1|, |k_2|, |k_3|) \label{Z2}
\end{align}

\noi
in the following, since this corresponds to the worst case.

\medskip
\noi
$\circ$ Case 2.a:  $\Phi(\bar{k})$  satisfies \eqref{PHI1}.
In this case, we have
\begin{align} \label{N26}
M_2^2 \lesssim \frac{1}{|j_2|^{2s}|j_3|^{2s} |k_2|^{2s}|k_3|^{2s} \ld^2},
\end{align}

\noi
where $\ld$ is as in \eqref{LD1}.
(Note that $\ld$ is of no help if $\ld = |k_2 + k_3|$ and $|k_1| \gg |j_2|, |j_3|\gg |k_2|, |k_3|$.)
Thus, for $ s >\frac{1}{2}$, we have
$\mathcal{M}_2 \lesssim 1$
by summing over $j_2, j_3, k_2, k_3$,
 and then $k$ for $z_k^2$.
 Hence, \eqref{N25} holds for $s > \frac{1}{2}$.

\medskip
\noi
$\circ$ Case 2.b: $\Phi(\bar{k})$ satisfies \eqref{PHI2}.
In this case, we have $ k^* \sim |k_1| \sim |k_2|\sim |k_3| \gtrsim |k|$.
Then, for $s \geq \frac{1}{2}$,
we have
\begin{align} \label{N27}
M_2^2 \lesssim \frac{|k_1|^2 }{(k^*)^{2s} |j_1|^{2s}|j_2|^{2s}|j_3|^{2s}\Ld^2}
\lesssim \frac{1}{ |j_2|^{2s}|j_3|^{2s}\Ld^2},
\end{align}

\noi
where $\Ld$ is as in \eqref{LD2}.
Thus, we have
$\mathcal{M}_2 \lesssim 1$
by summing over $k_2, k_3$ for $\Ld^{-2}$,
$j_2, j_3$,
 and then $k$ for $z_k^2$.
 Hence, \eqref{N25} holds for $s > \frac{1}{2}$.
\end{proof}

\begin{lemma} \label{LEM:N3}
Let $\N_1^{(-n)}$ be as in \eqref{v5}.
Then, for $s>0$, there exists $\al > 0$ such that
\begin{align} \label{N31}
\|\N_1^{(-n)}(v)\|_{H^s} & \lesssim  n^{-\al } \|v\|_{H^s}^3 \\
\label{N32}
\|\N_1^{(-n)}(v) - \N_1^{(-n)}(w) \|_{H^s}
& \lesssim   n^{-\al } \big(\|v\|_{H^s}+ \|w\|_{H^s}\big)^2 \|v-w\|_{H^s}.
\end{align}
\end{lemma}

\begin{proof}
We only prove \eqref{N31} since \eqref{N32} follows in a similar manner.
Recall that
\[ \N_1^{(-n)} =
\frac{i}{3}\sum_{\substack{k_1+ k_2+ k_3 = k\\\Phi(\bar{k}) \ne 0\\k^* >  n}}
\frac{k e^{i t\Phi(\bar{k})}}{i \Phi(\bar{k})} v_{k_1}v_{k_2}v_{k_3}.\]

\noi
By duality, it suffices to prove
\begin{align} \label{N33}
 \sum_k \sum_{\substack{k_1+ k_2+ k_3 = k\\\Phi(\bar{k}) \ne 0 \\k^* > n}}
M_3 u_{k_1}u_{k_2}u_{k_3} z_k
\lesssim n^{-\al} \|u\|_{L^2}^3
\end{align}

\noi
where $\|z\|_{L^2} = 1$ and $M_3 $ is given by
\[ M_3 = M_3(k_1, k_2, k_3) : = \frac{|k|^{1+s} }{|\Phi(\bar{k})||k_1|^{s} |k_2|^s|k_3|^s}.\]

\noi
As before, by Cauchy-Schwarz inequality, we have
\begin{align*}
\text{LHS of }\eqref{N33}
& \leq
\mathcal{M}_3
\|u\|_{L^2}^3,
\end{align*}

\noi
where $\mathcal{M}_3 = \big( \sum_k \sum_{k_1 + k_2 +k_3 = k}   M_3^2 z_k^2\big)^\frac{1}{2}$.
Without loss of generality, assume $k^* = |k_1|$.

\medskip
\noi
$\circ$ Case 1:  $\Phi(\bar{k})$  satisfies \eqref{PHI1}.
In this case, we have
\[ M_3^2 \lesssim \frac{1}{(k^*)^2 |k_2|^{2s}|k_3|^{2s} \ld^2}
\leq n^{-1}\frac{1}{|k_2|^{1+ 2s} \ld^2}, \]

\noi
where $\ld$ is as in \eqref{LD1}.
Thus, for $ s > 0 $, we have
$\mathcal{M}_3 \lesssim n^{-1}$
by summing over $k_j (\ne k_2)$ appearing in $\ld^{-2}$,
$k_2$,
 and then $k$ for $z_k^2$.
 Hence, \eqref{N33} holds for $s > 0$.

\medskip
\noi
$\circ$ Case 2:  $\Phi(\bar{k})$  satisfies \eqref{PHI2}.
For $\Ld$  as in \eqref{LD2},
we have $ \max(|k_2|, |k_3|, \Ld) \gtrsim |k_1|>n$,
since we have $|k_2|\sim |k_1|$ or $|k_3|\sim |k_1|$ if $\Ld \ll |k_1|$.
Then, we have
\[ M_3^2 \lesssim \frac{1}{ |k_2|^{2s}|k_3|^{2s} \Ld^2}
\lesssim \max(n^{-2s}, n^{-1+\eps}) \frac{1}{\Ld^{1+\eps}}, \]

\noi
Thus, for $ s > 0 $, we have
$\mathcal{M}_3 \lesssim n^{-\al}$
by summing over two frequencies for  $\Ld^{-1-\eps}$,
 and then $k$ for $z_k^2$.
 Hence, \eqref{N33} holds for $s > 0$.
\end{proof}

\section{Unconditional local well-posedness for $s > \frac{1}{2}$}
\label{SEC:LWP1}
In this section, we put together all the lemmata in the previous section
and prove unconditional local well-posedness of mKdV (with prescribed $L^2$-norm)
in $H^s(\T)$, $s >\frac{1}{2}$.
Some parts of the argument below are standard.
However, we include them for completeness.

For $n \in \mathbb{N}$, define $F_n(v, v_0)$ by $F_n(v, v_0) = F_n^{(1)}(v, v_0)+F_n^{(2)}(v, v_0)$,
where $F_n^{(1)}$ and $F_n^{(2)}$ are given by
\begin{align*}
F_n^{(1)}& = \N^{(-n)}_{1}(v)(t)-\N_{1}^{(-n)}(v_0)\\
F_n^{(2)}& = \int_0^t \RR(v)(t') + \N^{(n)}(v)(t')+ \N_{2}^{(-n)}(v)(t')dt',
\end{align*}

\noi
Then, if $v$ is a solution to \eqref{v1}, then we have
\begin{equation} \label{M1}
v(t) = v_0 + F_n(v, v_0)(t).
\end{equation}

Given an initial condition $v_0 \in H^s$, $s>\frac{1}{2}$,
take a sequence $\big\{v_0^{[m]}\big\}_{m \in \mathbb{N}}$ of smooth functions such that $v_0^{[m]} \to v_0$ in $H^s$.
Let $R = \|v_0\|_{H^s} + 1$.
Then, without loss of generality, we can assume
$\big\|v_0^{[m]}\|_{H^s} \leq R$.

Let $v^{[m]}$ denote the smooth global-in-time solution of mKdV with smooth initial condition
$v_0^{[m]}$.\footnote{For smooth initial data, there exists a global smooth solution thanks to either the theory of complete integrability or the energy method. Instead of using smooth solutions, we could directly construct a solution by Galerkin approximation
and compactness argument.
Here, we use smooth solutions for conciseness of the presentation.}
Then, by Lemmata \ref{LEM:R}--\ref{LEM:N3},
we have
\begin{align} \label{M2}
\|v^{[m]}\|_{C([0, T]; H^s)}
 \leq & R + C \big\{n^{-\al}  \big(\|v^{[m]}\|_{C([0, T]; H^s)}
 + \|v_0^{[m]}\|_{H^s}\big)^2 \notag \\
 & + n \ln n \,  T (\|v^{[m]}\|_{C([0, T]; H^s)}^2 +\|v^{[m]}\|_{C([0, T]; H^s)}^4) \big\}
\|v^{[m]}\|_{C([0, T]; H^s)}
\end{align}

\noi
First, choose $n$ sufficiently large such that
\begin{equation} \label{M21}
C n^{-\al} (3R)^2 < \tfrac{1}{4}.
\end{equation}

\noi
Next, choose $T$ sufficiently small
such that
\begin{equation} \label{M22}
C n \ln n \, T \big( (2R)^2 + (2R)^4\big) <\tfrac{1}{4}.
\end{equation}

\noi
Then, from \eqref{M2} with the continuity argument , we have
\[\|v^{[m]}\|_{C([0, T]; H^s)} \leq 2R, \quad m \in \mathbb{N}.\]

\noi
Moreover, we have
\begin{align} \label{M3}
\| F_n(v^{[m_1]}, v_0^{[m_1]}) & - F_n(v^{[m_2]}, v_0^{[m_2]})\|_{C([0, T]; H^s)} \notag\\
& \leq C_R
\|v^{[m_1]}- v^{[m_2]}\|_{C([0, T]; H^s)}
+ (1+ C_R)\|v_0^{[m_1]}-v_0^{[m_2]}\|_{H^s}
\end{align}

\noi
where $C_R <\frac{1}{2}$ (by possibly taking larger $n$ and smaller $T$.)
Since $v^{[m_j]}$ is a (smooth) solution with initial condition $v_0^{[m_j]}$,
it follows from \eqref{M3} that
\begin{align} \label{M4}
\|v^{[m_1]}- v^{[m_2]}\|_{C([0, T]; H^s)}
\leq C'\|v_0^{[m_1]}-v_0^{[m_2]}\|_{H^s}
\end{align}

\noi
for some $C' >0$.
Hence, $\{v^{[m]}\}$ converges in $C([0, T]; H^s)$.

Let $v^\infty$ denote the limit.
Then, we need to show that $v^\infty$ satisfies
\eqref{v1} or
\begin{equation} \label{M6}
v(t) = v_0 + \int_0^t \N(v)(t') + \RR(v) (t') dt'
\end{equation}

\noi
as a space-time distribution.
First, observe the following lemma.
We present the proof at the end of this section.
\begin{lemma} \label{LEM:H2}
Let $\N$ and $\RR$ be as in \eqref{v1}.
Then, we have, for any $\eps > 0$,
\begin{equation}\label{H21}
\|\N +\RR \|_{H^{-\frac{1}{2}-\eps}}
\lesssim \|v\|_{H^\frac{1}{2}}^3.
\end{equation}

\noi
In particular, if $v$ satisfies \eqref{v1},
then we have
$\|\dt v\|_{H^{-\frac{1}{2}-\eps}}
\lesssim \|v\|_{H^\frac{1}{2}}^3.$

\end{lemma}

Given a test function $\phi$,
consider
\begin{align}
& \int_0^T \int_\T \bigg\{ v^\infty(t)  - v_0 - \int_0^t \wt{\N}(v^\infty)(t') dt' \notag \\
& \hphantom{XXXX}
- \Big[ v^{[m]}(t) - v_0^{[m]} - \int_0^t \wt{\N}(v^{[m]})(t')   dt'
\Big] \bigg\} \phi(x, t) dx dt \notag \\
& \hphantom{XX}
=  \int_0^T  \int_\T (v^\infty(t) - v^{[m]}(t))\phi \, dx dt
- \int_0^T \int_\T (v_0 - v_0^{[m]}) \phi \, dx dt \notag \\
& \hphantom{XXXX}
+ \int_0^T \int_\T \int_0^t \Big[\wt{\N}(v^\infty)(t') - \wt{\N}(v^{[m]})(t')\Big] dt' \phi(x, t)dx dt \notag \\
&  \hphantom{XX}
=: I_1 - I_2 + I_3
\label{M7}
\end{align}

\noi
where $\wt{\N} = \N + \RR$.
By convergence of $v^{[m]} \to v^\infty$ in $C([0, T];H^s)$
and $v_0^{[m]} \to v_0$ in $H^s$, we have $I_1, I_2 \to 0$ as $m \to \infty$.
By Lemma \ref{LEM:H2},
we have
\begin{align*}
|I_3|
\lesssim T (\|v^\infty\|_{C([0, T];H^s)}^2 + \|v^{[m]}\|_{C([0, T];H^s)}^2 )
\|v^\infty - v^{[m]}\|_{C([0, T];H^s)} \|\phi\|_{L^1_t H^{\frac{1}{2}+\eps}_x}
\to 0
\end{align*}

\noi
as $m\to \infty$.
Therefore, $v^\infty$ is a solution to \eqref{M6}.
It follows from  \eqref{M21} and \eqref{M22}
that the time of existence $T$ satisfies $T \gtrsim \|v_0\|_{H^s}^{-\beta}$
for some $\beta > 0$.
Also, the Lipschitz dependence on initial data follows from \eqref{M4}.

Let $T$ be given.
Suppose that both $v$ and $\wt{v}$ are  solutions in $C([0, T]; H^s)$ to \eqref{v1}
with the same initial condition $v_0 \in H^s(\T)$, $s >\frac{1}{2}$.
First, assume that $\|v\|_{C([0, T]; H^s)}, \|\wt{v}\|_{C([0, T]; H^s)} \leq 2R$
where $R = \|v_0\|_{H^s} + 1$.
Choose $n$ and $\tau$ satisfying \eqref{M21} and \eqref{M22}
(in place of $T$.)
Then, from \eqref{M3}, we have
\begin{align*}
\| v -\wt{v}\|_{C([0, \tau]; H^s)}
= \| F_n(v, v_0)  - F_n(\wt{v}, v_0)\|_{C([0, \tau]; H^s)}
 \leq \tfrac{1}{2}
\|v- \wt{v}\|_{C([0, \tau]; H^s)}.
\end{align*}

\noi
Hence, $v = \wt{v}$ in $C([0, \tau]; H^s)$.
By iterating the argument, we obtain $v = \wt{v}$ in $C([0, T]; H^s)$.
Now, suppose that $\wt{R}:= \frac{1}{2}\max(\|v\|_{C([0, T]; H^s)}, \|\wt{v}\|_{C([0, t]; H^s)}) > R$.
Then, use $\wt{R}$ (in place of $R$) to
determine $n$ and $\tau$ (in place of $T$) in \eqref{M21} and \eqref{M22}.
The rest follows as before.
This proves the unconditional uniqueness (with prescribed $L^2$-norm.)

We conclude this section by presenting the proof of Lemma \ref{LEM:H2}.
\begin{proof} [Proof of Lemma \ref{LEM:H2}]
The contribution from $\RR$ is bounded by Lemma \ref{LEM:R}.
Without loss of generality, assume $|k_1| = \max(|k_1|, |k_2|, |k_3|) \gtrsim |k|$.
Then, by Young's inequality, we have
\begin{align*}
\|\N\|_{H^{-\frac{1}{2}-\eps}}
& \lesssim
\bigg(\sum_k  \Big( \sum_{k_1+k_2+k_3=k}
|k_1|^{\frac{1}{2}} v_{k_1}
|k_2|^{-\frac{1}{2}\eps}v_{k_2}|k_2|^{-\frac{1}{2}\eps}v_{k_3} \Big)^2 \bigg)^\frac{1}{2}\\
& \leq  \|v\|_{H^\frac{1}{2}}
\big\||k|^{-\frac{1}{2}\eps} v_{k}\big\|_{l^1_{k}}^2
\lesssim  \|v\|_{H^\frac{1}{2}} ^3.
\end{align*}

\noi
where we used Cauchy-Schwarz inequality in the last step.
\end{proof}

\section{Endpoint case: $s = \frac{1}{2}$} \label{SEC:endpoint}
The previous argument fails at the endpoint regularity $ s= \frac{1}{2}$ precisely because
$\N_{2}^{(-n)}$ does not satisfy the required estimate when $ s= \frac{1}{2}$.
See Lemma \ref{LEM:N2}.
However, when $s= \frac{1}{2}$,
Lemma \ref{LEM:N2} still holds for
$\N_{21}^{(-n)}$ defined in \eqref{N23}.
Moreover, if any of the following conditions holds, then
Lemma \ref{LEM:N2} holds for
$\N_{22}^{(-n)}$ defined in \eqref{N23}, even when $ s= \frac{1}{2}$:
\begin{eqnarray*}
\textup{(a)}  &
\max(|k_2|, |k_3|) &\gtrsim \min( |k|^\frac{1}{100}, |k_1|^\frac{1}{100}),
\text{ if }\Phi(\bar{k}) \text{ satisfies }\eqref{PHI1},\\
\textup{(b)} &
\max(|k_2|, |k_3|)& \gtrsim\min(|j_2|^\frac{1}{100}, |j_3|^\frac{1}{100}),
\text{ if }\Phi(\bar{k}) \text{ satisfies }\eqref{PHI1},\\
 \textup{(c)}& |j_1|& \gtrsim \min(|k|^{1+\frac{1}{100}}, |k_1|^{1+\frac{1}{100}}), \\
 \textup{(d)}& \Ld_j& \lesssim \max(|k_2|^\frac{1}{100},  |k_3|^\frac{1}{100}), \\
 &&\hspace{+50pt}\vphantom{\Big|}
 \text{if } \Phi(\bar{k}) \text{ satisfies }\eqref{PHI1}
 \text{ and }\Phi(\bar{j}) \text{ satisfies }\eqref{PHI2}, \\
 \textup{(e)}& \Ld_k& \gtrsim \max(|j_2|^\frac{1}{100},  |j_3|^\frac{1}{100}), \text{ if }\Phi(\bar{k}) \text{ satisfies }\eqref{PHI2},
\end{eqnarray*}

\noi
where $\Ld_j$  is as in \eqref{LD2} with $\{j_i\}$ in place of $\{k_i\}$.

\medskip
By symmetry,  assume $|k_2| \geq |k_3|$ and $|j_2| \geq |j_3|$.
First, suppose that (a) holds.
Then,
in \eqref{N26}, we have
\begin{align} \label{O11}
M_2^2 \lesssim \frac{1}{|j_1|^{2\eps} |j_2|^{2s}|j_3|^{2s} |k_2|^{2s-200\eps}|k_3|^{2s} \ld^2}
\lesssim \frac{1}{ |j_2|^{2s+\eps}|j_3|^{2s+\eps} |k_2|^{2s-201\eps}|k_3|^{2s+\eps} \ld^2}.
\end{align}

\noi
Thus, for $ s  = \frac{1}{2}$, we have
$\mathcal{M}_2 \lesssim 1$
by summing over $k_i \,(\ne k_3)$ for $\ld^{-2}$, $k_3$,
$j_2, j_3$,
and then $k$ for $z_k^2$.
Hence, \eqref{N25} holds for $s = \frac{1}{2}$.
Secondly, suppose that (b) holds.
By $|k_2|^{200\eps} \gtrsim |j_2|^\eps|j_3|^\eps$,
we have $M_2^2 \lesssim$ RHS of \eqref{O11}
as before.
%
Next, suppose that (c) holds.
Then, we have a small additional power of $|j_1|$ in the denominators
of \eqref{N26} and \eqref{N27}
Hence, \eqref{N25} holds for $s = \frac{1}{2}$.

Now, suppose (d) holds.
This implies that either
\begin{eqnarray*}
\text{(d.1)} &  |j_2 + j_3|& \lesssim \max(|k_2|^\frac{1}{100},  |k_3|^\frac{1}{100}),
\text{ or } \\
\text{(d.2)} &
|j_1 + j_2||j_1 + j_3|& \lesssim \max(|k_2|^\frac{1}{100},  |k_3|^\frac{1}{100})
\end{eqnarray*}

\noi
By symmetry, assume $|k_2| \geq |k_3|$.
If (d.1) holds, then for fixed $j_3$, there are at most $O\big(|k_2|^\frac{1}{100}\big)$
possible choices for $j_2$.
Then, by going back to Case 2.a in Lemma \ref{LEM:N2}, we
have
\[
M_2^2 \lesssim \frac{1}{|j_2|^{2s}|j_3|^{2s} |k_2|^{2s}|k_3|^{2s} \ld^2},
\lesssim \frac{1}{|j_3|^{4s} |k_2|^\frac{1}{100}|k_3|^{4s-\frac{1}{100}} \ld^2}.
\]

\noi
Thus, for $ s >\frac{101}{400}$, we have
$\mathcal{M}_2 \lesssim 1$
by summing over $k_i \, (\ne k_3)$ for $\ld^{-2}$,
$k_3$, then, $j_2$ and $j_3$,
 and finally $k$ for $z_k^2$.
 Hence, \eqref{N25} holds for $s = \frac{1}{2}$.

If (d.2) holds, then
then for fixed $j_1$, there are at most $O\big(|k_2|^\frac{1}{100}\big)$
possible choices for $j_3$.
Since  $\Phi(\bar{\jmath})$ satisfies \eqref{PHI2}, we have
$|j_1|\sim |j_2|$.
Then, by going back to Case 2.a in Lemma \ref{LEM:N2}, we
have
\[
M_2^2 \lesssim \frac{1}{|j_1|^{2s}|j_3|^{2s} |k_2|^{2s}|k_3|^{2s} \ld^2},
\lesssim \frac{1}{|j_1|^{4s} |k_2|^\frac{1}{100}|k_3|^{4s-\frac{1}{100}} \ld^2}.
\]

\noi
Once again, for $ s >\frac{101}{400}$, we have
$\mathcal{M}_2 \lesssim 1$
by summing over $k_i \, (\ne k_3)$ for $\ld^{-2}$,
$k_3$, then, $j_3$ and $j_1$,
 and finally $k$ for $z_k^2$.
 Hence, \eqref{N25} holds for $s = \frac{1}{2}$.

Finally, suppose (e) holds. Then, in \eqref{N27}, we have
\[M_2^2 \lesssim \frac{1}{|j_2|^{2s+\eps}|j_3|^{2s+\eps}
\Ld_k^{2-200\eps}}.\]

\noi
\noi
Thus, for $ s =\frac{1}{2}$, we have
$\mathcal{M}_2 \lesssim 1$
by summing over two frequencies for $\Ld_k$,
then, $j_2$ and $j_3$,
 and finally $k$ for $z_k^2$.
 Hence, \eqref{N25} holds for $s = \frac{1}{2}$.

Hence, letting $\N_{221}^{(-n)}$ be the restriction of $\N_{22}^{(-n)}$
such that at least one of the above conditions (a), (b), (c), (d), or (e) holds,
we have the following estimates.
\begin{lemma} \label{LEM:O1}
There exists $s_0 <\frac{1}{2}$ such that
the following estimates hold for $s >s_0$:
\begin{align} \label{O111}
\|\N_{221}^{(-n)}(v)\|_{H^s} & \lesssim  \, \|v\|_{H^s}^5 \\
\label{O112}
\|\N_{221}^{(-n)}(v) - \N_{221}^{(-n)}(w) \|_{H^s}
& \lesssim   \big(\|v\|_{H^s}+ \|w\|_{H^s}\big)^4 \|v-w\|_{H^s}.
\end{align}
\end{lemma}

Now, letting $\N_{222}^{(-n)} := \N_{22}^{(-n)} - \N_{221}^{(-n)}$,
we have
\begin{equation} \label{O13}
\N_{22}^{(-n)} = \N_{221}^{(-n)} + \N_{222}^{(-n)}.
\end{equation}

\noi
In the following, we concentrate on estimating
the contribution from $\N_{222}^{(-n)}$.
Note that $\N_{222}^{(-n)}$ is the restriction of $\N_{22}^{(-n)}$
such that {\it all} of the conditions below hold:
\begin{eqnarray}
\textup{(a')}  &
\max(|k_2|, |k_3|) &\ll \min( |k|^\frac{1}{100}, |k_1|^\frac{1}{100}),  \label{O14}\\
\textup{(b')} &
\max(|k_2|, |k_3|)& \ll\min(|j_2|^\frac{1}{100}, |j_3|^\frac{1}{100}),\label{O15}\\
 \textup{(c')}& |j_1|& \ll \min(|k|^{1+\frac{1}{100}}, |k_1|^{1+\frac{1}{100}}), \label{O16}\\
 \text{(d')} & \Ld_j & \gg \max(|k_2|^\frac{1}{100},  |k_3|^\frac{1}{100}) \label{O166},\\
\text{(e')} & \Ld_k & \ll \max(|j_2|^\frac{1}{100},  |j_3|^\frac{1}{100}), \label{O167}
\end{eqnarray}

\noi
where (a') and (b') hold when $\Phi(\bar{k})$ satisfies \eqref{PHI1},
(d') holds when $\Phi(\bar{k})$ satisfies \eqref{PHI1} and $\Phi(\bar{\jmath})$ satisfies \eqref{PHI2},
and (e') holds when $\Phi(\bar{k})$ satisfies \eqref{PHI2}.
(Recall that we also assume \eqref{Z1}--\eqref{Z2}.)
Henceforth, we assume that the frequencies are restricted such that
the conditions (a')--(e') hold.
By \eqref{v1} and \eqref{v2}, we have
\begin{align}
\notag
\big(\N_{222}^{(-n)}\big)_k
& = \frac{1}{3}\sum_{\substack{k_1+ k_2+ k_3 = k \\
j_1 + j_2 +j_3 = k_1 \\ \Phi(\bar{k}), \Phi(\bar{\jmath}) \ne 0\\k^* > n}}
\frac{k k_1e^{i t(\Phi(\bar{k})+\Phi(\bar{\jmath}))}}{i \Phi(\bar{k})} v_{j_1}v_{j_2}v_{j_3}v_{k_2}v_{k_3}  \\
& = -i \sum_{\substack{k_1+ k_2+ k_3 = k\\
\Phi(\bar{k})\ne 0\\k^* > n}}
\frac{k e^{i t(\Phi(\bar{k}))}}{i \Phi(\bar{k})} \dt (\N_1)_{k_1} v_{k_2}v_{k_3}
+i \sum_{\substack{k_1+ k_2+ k_3 = k\\
\Phi(\bar{k})\ne 0\\k^* > n}}
\frac{k e^{i t(\Phi(\bar{k}))}}{i \Phi(\bar{k})}  (\N_2)_{k_1} v_{k_2}v_{k_3} \notag\\
& = : (\N_3)_k +(\N_4)_k.  \label{O17}
\end{align}

The following lemma shows that $\N_4$ can be controlled in $H^\frac{1}{2}$.
\begin{lemma} \label{LEM:O2}
The following estimates hold:
\begin{align} \label{O221}
\|\N_4(v)\|_{H^\frac{1}{2}} & \lesssim  \, \|v\|_{H^\frac{1}{2}}^7 \\
\label{O222}
\|\N_4(v) - \N_4(w) \|_{H^\frac{1}{2}}
& \lesssim   \big(\|v\|_{H^\frac{1}{2}}+ \|w\|_{H^\frac{1}{2}}\big)^6 \|v-w\|_{H^\frac{1}{2}}.
\end{align}
\end{lemma}

\noi
Before proving this lemma,
let us present the following corollary to Lemma \ref{LEM:N2}.
\begin{corollary} \label{COR:H1}
For $s < \frac{1}{2}$, we have
\begin{align} \label{H11}
\|\N_2(v)\|_{H^s} & \lesssim  \, \|v\|_{H^\frac{1}{2}}^5 \\
\label{H12}
\|\N_2(v) - \N_2(w) \|_{H^s}
& \lesssim   \big(\|v\|_{H^\frac{1}{2}}+ \|w\|_{H^\frac{1}{2}}\big)^4 \|v-w\|_{H^\frac{1}{2}}.
\end{align}
\end{corollary}

\noi
We omit the proof of this corollary, since it follows from a slight modification of
the proof of Lemma \ref{LEM:N2}.

\begin{proof}[Proof of Lemma \ref{LEM:O2}]
We only prove \eqref{O221} since \eqref{O222} follows in a similar manner.
In view of Corollary \ref{COR:H1}, it suffices to prove
\begin{equation}
\|\N_4\|_{H^\frac{1}{2}}
\lesssim \|\N_2\|_{H^{-\frac{1}{2}}} \|v\|_{H^\frac{1}{2}}^2.
\end{equation}

\noi
By duality, it suffices to prove
\begin{align} \label{O223}
 \sum_k \sum_{\substack{k_1+ k_2+ k_3 = k\\\Phi(\bar{k}) \ne 0 \\k^* > n}}
M_4 u_{k_1}u_{k_2}u_{k_3} z_k
\lesssim  \|u\|_{L^2}^3
\end{align}

\noi
where $\|z\|_{L^2} = 1$ and $M_4 $ is given by
\[ M_4 = M_4(k_1, k_2, k_3) : = \frac{|k|^\frac{3}{2}|k_1|^\frac{1}{2}   }{|\Phi(\bar{k})||k_2|^\frac{1}{2}|k_3|^\frac{1}{2}}.\]

\noi
As before, by Cauchy-Schwarz inequality, we have
$\text{LHS of }\eqref{O223}
 \leq
\mathcal{M}_4
\|u\|_{L^2}^3$,
where $\mathcal{M}_4 = \big( \sum_k \sum_{k_1 + k_2 +k_3 = k}   M_4^2 z_k^2\big)^\frac{1}{2}$.

\medskip
\noi
$\circ$ Case 1:  $\Phi(\bar{k})$  satisfies \eqref{PHI1}.
In this case, we have
$ M_4^2 \lesssim |k_2|^{-1}|k_3|^{-1} \ld^{-2}
\leq |k_3|^{-2} \ld^{-2}$,
where $\ld$ is as in \eqref{LD1}.
Thus, we have $\mathcal{M}_4 \lesssim 1$
by summing over $k_i \, (\ne k_3)$ appearing in $\ld^{-2}$,
$k_3$,
 and then $k$ for $z_k^2$.
 Hence, \eqref{O223} holds.

\medskip
\noi
$\circ$ Case 2:  $\Phi(\bar{k})$  satisfies \eqref{PHI2}.
By $|k_2|\sim|k_1| \sim k^*$, we have
$M_4^2 \lesssim   \Ld^{-2}$
where $\Ld$ is as in \eqref{LD2}.
Thus, we have
$\mathcal{M}_4 \lesssim 1$
by summing over two frequencies for  $\Ld^{-2}$,
 and then $k$ for $z_k^2$.
 Hence, \eqref{O223} holds.
\end{proof}

Now, we need to estimate the contribution from $\N_3$.
In doing so, we actually estimate $\int_0^t \N_3(t')dt'$.
Integrating by parts, we have
\begin{align}
\int_0^t  \N_3(t')dt'
& = -\frac{1}{3}\int_0^t
\sum_{\substack{k_1+ k_2+ k_3 = k \\
j_1 + j_2 +j_3 = k_1 \\ \Phi(\bar{k}), \Phi(\bar{\jmath}) \ne 0\\k^* > n}}
\frac{k k_1e^{i t'(\Phi(\bar{k})+\Phi(\bar{\jmath}))}}{i \Phi(\bar{\jmath})} \big(v_{j_1}v_{j_2}v_{j_3}v_{k_2}v_{k_3}\big) (t') dt'\notag \\
& +\frac{1}{3} \int_0^t \sum_{\substack{k_1+ k_2+ k_3 = k \\
j_1 + j_2 +j_3 = k_1 \\ \Phi(\bar{k}), \Phi(\bar{\jmath}) \ne 0\\k^* > n}}
\frac{k k_1e^{i t'(\Phi(\bar{k})+\Phi(\bar{\jmath}))}}{ \Phi(\bar{k})\Phi(\bar{\jmath})} \big(v_{j_1}v_{j_2}v_{j_3}\big)(t')
\dt\big(v_{k_2}v_{k_3}\big) (t') dt' \notag\\
& -\frac{1}{3}  \sum_{\substack{k_1+ k_2+ k_3 = k \\
j_1 + j_2 +j_3 = k_1 \\ \Phi(\bar{k}), \Phi(\bar{\jmath}) \ne 0\\k^* > n}}
\frac{k k_1e^{i t(\Phi(\bar{k})+\Phi(\bar{\jmath}))}}{\Phi(\bar{k})\Phi(\bar{\jmath})} \big(v_{j_1}v_{j_2}v_{j_3}
v_{k_2}v_{k_3}\big)(t')\bigg|_0^t \notag\\
& =: \int_0^t \N_5 (t') dt'+ \int_0^t \N_6 (t') dt' + \N_7(v)(t) - \N_7(v)(0). \label{O333}
\end{align}

First, note that $\N_5$ looks like $\N_2^{(-n)}$ in Lemma \ref{LEM:N2}.
However, it satisfies a better estimate thanks to the conditions (b')--(e').
\begin{lemma} \label{LEM:O3}
The following estimates hold:
\begin{align} \label{O31}
\|\N_5(v)\|_{H^\frac{1}{2}} & \lesssim  \, \|v\|_{H^\frac{1}{2}}^5 \\
\label{O32}
\|\N_5(v) - \N_5(w) \|_{H^\frac{1}{2}}
& \lesssim   \big(\|v\|_{H^\frac{1}{2}}+ \|w\|_{H^\frac{1}{2}}\big)^4 \|v-w\|_{H^\frac{1}{2}}.
\end{align}
\end{lemma}

\begin{proof}
We only prove \eqref{O31} since \eqref{O32} follows in a similar manner.
By duality, it suffices to prove
\begin{align} \label{O33}
\sum_k
\sum_{\substack{k_1+ k_2+ k_3 = k \\
j_1 + j_2 +j_3 = k_1 \\ \Phi(\bar{k}), \Phi(\bar{\jmath}) \ne 0\\k^* > n}}
M_5 u_{j_1}u_{j_2}u_{j_3}u_{k_2}u_{k_3} z_k
\lesssim  \|u\|_{L^2}^5
\end{align}

\noi
where $\|z\|_{L^2} = 1$ and $M_5 $ is given by
\[ M_5 = M_5(j_1,j_2,j_3, k_2, k_3) : =
\frac{|k|^{1+s}|k_1|   }{|\Phi(\bar{\jmath})||j_1|^s|j_2|^s|j_3|^s|k_2|^s|k_3|^s}\]

\noi
with $s = \frac{1}{2}$.
As before, by Cauchy-Schwarz inequality, we have
$\text{LHS of }\eqref{O33}
 \leq
\mathcal{M}_5
\|u\|_{L^2}^5$,
where $\mathcal{M}_5 = \big( \sum_{\substack{k_1+ k_2+ k_3 = k \\
j_1 + j_2 +j_3 = k_1}}
   M_5^2 z_k^2\big)^\frac{1}{2}$.

By symmetry, assume $|k_2| \geq |k_3|$ and $|j_2| \geq |j_3|$.

\medskip
\noi
$\circ$ Case 1:  Both $\Phi(\bar{k})$  and $\Phi(\bar{\jmath})$  satisfy \eqref{PHI1}.
Let $\ld_j$ be as in \eqref{LD1} with $\{j_i\}$ in place of $\{k_i\}$.
Then, as in Case 2.a in Lemma \ref{LEM:N2}, we have
\[M_5^2 \lesssim \frac{1}{|j_2|^{2s}|j_3|^{2s} |k_2|^{2s}|k_3|^{2s} \ld_j^2}
\ll \frac{1}{|j_3|^{4s-\frac{2\eps}{100}} |k_2|^{2s+\eps}|k_3|^{2s+\eps} \ld_j^2},
\]

\noi
since $|j_2| \gg |k_2|^{100} \geq |k_3|^{100}$ by the condition (b').
Thus, for $s = \frac{1}{2}$, we have
$\mathcal{M}_5 \lesssim 1$
by summing over $k_2, k_3$, then $j_i \,(\ne j_3)$, for $\ld_j^{-2}$,
$j_3$,
 and $k$ for $z_k^2$.
 Hence, \eqref{O33} holds for $s = \frac{1}{2}$.

\medskip
\noi
$\circ$ Case 2:  $\Phi(\bar{k})$ satisfies \eqref{PHI2} and
$\Phi(\bar{\jmath})$  satisfies \eqref{PHI1}.
By (e'),  given $k_3$, there are at most $O(|j_2|^\frac{1}{50})$ possible choices for $k_2$.
Then, by (c') with $|k_2|\sim|k_1|$,
given $k_3$, there are at most $O(|k_2|^\frac{101}{5000})$ possible choices for $k_2$.
As in Case 2.a in Lemma \ref{LEM:N2}, we have
\[M_5^2 \lesssim \frac{1}{|j_2|^{2s}|j_3|^{2s} |k_2|^{2s}|k_3|^{2s} \ld_j^2}
\ll \frac{1}{|j_3|^{4s} |k_2|^{\frac{1}{40}}|k_3|^{4s-\frac{1}{40}} \ld_j^2}.\]

\noi
Thus, for $s = \frac{1}{2}$, we have
$\mathcal{M}_5 \lesssim 1$
by summing over $k_2, k_3$, then $j_i \,(\ne j_3)$, for $\ld_j^{-2}$,
$j_3$,
 and $k$ for $z_k^2$.
 Hence, \eqref{O33} holds for $s = \frac{1}{2}$.

\medskip
\noi
$\circ$ Case 3: $\Phi(\bar{k})$ satisfies \eqref{PHI1} and $\Phi(\bar{\jmath})$ satisfies \eqref{PHI2}.
In this case, we have $ j^* \sim |j_1| \sim |j_2|\sim |j_3| \gtrsim |k_1|$.
By (d') in \eqref{O166},
we have
\[M_5^2 \lesssim \frac{1}{ |k_2|^{2s}|k_3|^{2s}\Ld_j^2}
\ll \frac{1}{ |k_2|^{2s+\eps}|k_3|^{2s+\eps}\Ld_j^{2-200\eps}}\]

\noi
for $s \geq \frac{1}{2}$.
Thus, we have
$\mathcal{M}_2 \lesssim 1$
by summing over  two frequencies for $\Ld_j^{-2+200\eps}$,
$k_2, k_3$,
 and then $k$ for $z_k^2$.
 Hence, \eqref{O33} holds for $s = \frac{1}{2}$.

\medskip
\noi
$\circ$ Case 4: Both $\Phi(\bar{k})$  and $\Phi(\bar{\jmath})$ satisfy \eqref{PHI2}.
As in Case 2, given $k_3$,
there are at most $O(|k_2|^\frac{1}{40})$ possible choices for $k_2$.
Thus, we have
\[M_5^2 \lesssim \frac{1}{ |k_2|^{2s}|k_3|^{2s}\Ld_j^2}
\ll \frac{1}{ |k_2|^{\frac{1}{40}}|k_3|^{4s-\frac{1}{40}}\Ld_j^{2-200\eps}}\]

\noi
for $s \geq \frac{1}{2}$.
 Hence, \eqref{O33} holds for $s = \frac{1}{2}$.
\end{proof}

\begin{lemma} \label{LEM:O4}
Suppose that  $v$ and $w$ satisfy \eqref{v1}.
Then, the following estimates hold:
\begin{align} \label{O41}
\|\N_6(v)\|_{H^\frac{1}{2}} & \lesssim  \, \|v\|_{H^\frac{1}{2}}^7 \\
\label{O42}
\|\N_6(v) - \N_6(w) \|_{H^\frac{1}{2}}
& \lesssim   \big(\|v\|_{H^\frac{1}{2}}+ \|w\|_{H^\frac{1}{2}}\big)^6 \|v-w\|_{H^\frac{1}{2}}.
\end{align}
\end{lemma}

\begin{proof} 
Assume that the time derivative falls on $v_{k_2}$ in \eqref{O333}.
i.e. we have
\[ \N_6(t) \sim \sum_{\substack{k_1+ k_2+ k_3 = k \\
j_1 + j_2 +j_3 = k_1 \\ \Phi(\bar{k}), \Phi(\bar{\jmath}) \ne 0\\k^* > n}}
\frac{k k_1e^{i t'(\Phi(\bar{k})+\Phi(\bar{\jmath}))}}{ \Phi(\bar{k})\Phi(\bar{\jmath})} \big(v_{j_1}v_{j_2}v_{j_3}\big)
\dt v_{k_2} v_{k_3}. \]

\noi
Moreover, assume $|j_2| \geq |j_3|$ and $|k_2| \geq |k_3|$.
We only prove \eqref{O41} since \eqref{O42} follows in a similar manner.
In view of Lemma \ref{LEM:H2}, it suffices to prove
\begin{equation}
\|\N_6\|_{H^\frac{1}{2}}
\lesssim \|\dt v \|_{H^{-\frac{1}{2}-\eps}} \|v\|_{H^\frac{1}{2}}^4
\end{equation}

\noi
for some small $\eps >0$.
By duality, it suffices to prove
\begin{align} \label{O43}
 \sum_{\substack{k_1+ k_2+ k_3 = k \\
j_1 + j_2 +j_3 = k_1 \\ \Phi(\bar{k}), \Phi(\bar{\jmath}) \ne 0\\k^* > n}}
M_6 u_{j_1}u_{j_2}u_{j_3}u_{k_2}u_{k_3} z_k
\lesssim  \|u\|_{L^2}^5
\end{align}

\noi
where $\|z\|_{L^2} = 1$ and $M_6 $ is given by
\[ M_6 = M_6(j_1,j_2,j_3, k_2, k_3) : =
\frac{|k|^\frac{3}{2} |k_1| |k_2|^{\frac{1}{2}+\eps} }
{|\Phi(\bar{k})||\Phi(\bar{\jmath})||j_1|^\frac{1}{2}|j_2|^\frac{1}{2}|j_3|^\frac{1}{2}|k_3|^\frac{1}{2}}.\]

\noi
As before, by Cauchy-Schwarz inequality, we have
$\text{LHS of }\eqref{O43}
 \leq
\mathcal{M}_6
\|u\|_{L^2}^5$,
where $\mathcal{M}_6 = \big( \sum_{\substack{k_1+ k_2+ k_3 = k \\
j_1 + j_2 +j_3 = k_1 } }
 M_6^2 z_k^2\big)^\frac{1}{2}$.
Let $\ld_k$ and  $\ld_j$ (and $\Ld_k$ and  $\Ld_j$) be as in \eqref{LD1}
(and as in \eqref{LD2}) for $\{k_i\}$ and $\{j_i\}$, respectively.

\medskip

\noi
$\circ$ Case 1:  Both $\Phi(\bar{k})$ and $\Phi(\bar{\jmath})$ satisfy \eqref{PHI1}.
By (b') and $|j_1| \geq \max(|j_2|, |j_3|)$, we have
\[ M_6^2 \lesssim
\frac{1 }
{|j_1|^{4-\frac{1}{50}(\frac{1}{2}+\eps)}|j_2||j_3||k_3|\ld_k^2 \ld_j^2}
\leq  \frac{1 }
{|j_1|^{1+\eps}|j_2|^{1+\eps}|j_3|^{1+\eps}|k_3|^{1+\eps}\ld_k^2 \ld_j^2}
\]

\noi
for sufficiently small $\eps>0$.
Thus, we have
$\mathcal{M}_6 \lesssim 1$
by summing over $k_3, j_1,j_2,j_3$,
 and then $k$ for $z_k^2$.
Hence, \eqref{O43} holds for $s = \frac{1}{2}$.

\medskip

\noi
$\circ$ Case 2:  $\Phi(\bar{k})$ satisfies \eqref{PHI1}
and $\Phi(\bar{\jmath})$ satisfies \eqref{PHI2}.
In this case, we have
\begin{equation*}
M_6^2 \lesssim
\frac{1 }
{|j_1|^{2-\frac{1}{50}(\frac{1}{2}+\eps)}|j_2||j_3||k_3|\ld_k^2 \Ld_j^2}
\leq  \frac{1 }
{|j_1|^{1+\eps}|j_2|^{1+\eps}|j_3|^{1+\eps}|k_3|^{1+\eps}\ld_k^2 \Ld_j^2}
\end{equation*}

\noi
for sufficiently small $\eps>0$.
Thus, we have
$\mathcal{M}_6 \lesssim 1$
by summing over $k_3, j_1,j_2,j_3$,
 and then $k$ for $z_k^2$.
Hence, \eqref{O43} holds for $s = \frac{1}{2}$.

\medskip

\noi
$\circ$ Case 3:  $\Phi(\bar{k})$ satisfies \eqref{PHI2}
and $\Phi(\bar{\jmath})$ satisfies \eqref{PHI1}.
In this case, we have
\begin{equation*}
M_6^2 \lesssim
\frac{1}
{|j_1||j_2||j_3||k_3|^{1-2\eps} \Ld_k^2 \ld_j^2}
\leq
\frac{1}
{|j_2|^\frac{3}{2}|j_3|^\frac{3}{2}|k_3|^{1-2\eps} \Ld_k^2 \ld_j^2}.
\end{equation*}

\noi
Thus, we have
$\mathcal{M}_6 \lesssim 1$
by summing over two frequencies for $\Ld_k^{-2}$,
$j_2, j_3$,
 and then $k$ for $z_k^2$.
Hence, \eqref{O43} holds for $s = \frac{1}{2}$.

\medskip

\noi
$\circ$ Case 4:  Both $\Phi(\bar{k})$ and $\Phi(\bar{\jmath})$ satisfy \eqref{PHI2}.
In this case, we have $|k_1|\sim|k_2|\sim|k_3|$
and $|j_1|\sim|j_2|\sim|j_3|$.
In this case, we have
$M_6^2 \lesssim |j_1|^{-1+\eps} \Ld_k^{-2}\Ld_j^{-2}$
Thus, we have
$\mathcal{M}_6 \lesssim 1$
by summing over two frequencies for $\Ld_k^{-2}$,
two frequencies for $\Ld_j^{-2}$,  and then $k$ for $z_k^2$.
Hence, \eqref{O43} holds for $s = \frac{1}{2}$.
\end{proof}

Before we present the estimate on $\N_7$,
recall that  we have $|k_1| > n$.
Thus, we have $|j_1| \gtrsim n$ since $|j_1| \gtrsim |k_1|$.

\begin{lemma} \label{LEM:O5}
There exists $\al > 0$ such that
\begin{align} \label{O51}
\|\N_7(v)\|_{H^\frac{1}{2}} & \lesssim  n^{-\al} \, \|v\|_{H^\frac{1}{2}}^5 \\
\label{O52}
\|\N_7(v) - \N_7(w) \|_{H^\frac{1}{2}}
& \lesssim  n^{-\al} \, \big(\|v\|_{H^\frac{1}{2}}+ \|w\|_{H^\frac{1}{2}}\big)^4 \|v-w\|_{H^\frac{1}{2}}.
\end{align}
\end{lemma}

\begin{proof}
This lemma immediately follows from the proof of Lemma \ref{LEM:O4}
once we note that there is an extra (small) power of $|j_1|$ in the denominator
except for Case 3. In Case 3, we have an extra (small) power of $|k_3| \sim |k_1| > n$.
\end{proof}

The remaining part of the argument is basically the same as in Section \ref{SEC:LWP1}.
For $n \in \mathbb{N}$,
define $G_n(v, v_0)$ by $G_n(v, v_0) = G_n^{(1)}(v, v_0)+G_n^{(2)}(v, v_0)$,
where $G_n^{(1)}$ and $G_n^{(2)}$ are given by
\begin{align*}
G_n^{(1)} = & \, \N^{(-n)}_{1}(v)(t)-\N_{1}^{(-n)}(v_0)
+ \N_7(v)(t)-\N_7(v_0)\\
G_n^{(2)} = & \,  \int_0^t \RR(v)(t') + \N^{(n)}(v)(t')+ \N_{21}^{(-n)}(v)(t') \notag \\
& \hphantom{XXX}+ \N_{221}^{(-n)}(v)(t')+ \N_4(v)(t')+\N_5(v)(t')+\N_6(v)(t')dt'.
\end{align*}

\noi
From \eqref{v6}, \eqref{N23}, \eqref{O13}, \eqref{O17}, and \eqref{O333},
we have
\begin{align*}
v(t)  = v_0 +  G_n(v, v_0)(t).
\end{align*}

\noi
if $v$ is a solution to \eqref{v1}.
It follows from Lemmata \ref{LEM:R}--\ref{LEM:N3},
\ref{LEM:O1},\ref{LEM:O2}, and \ref{LEM:O3}--\ref{LEM:O5}
that we obtain a priori bounds \eqref{M2} and \eqref{M3}
for $s = \frac{1}{2}$
(with $G_n$ in place of $F_n$
and with a higher power in \eqref{M2}.)
Then, the unconditional local well-posedness in $H^\frac{1}{2}$
with Lipschitz dependence on initial data
follows
as in Section \ref{SEC:LWP1}.

\end{document}